\documentclass[a4paper]{amsart}

\usepackage{amssymb}
\usepackage{hyperref}
\hypersetup{colorlinks=true,
linkcolor=blue,
filecolor=purple, 
urlcolor=cyan,
citecolor=magenta
}

\newtheorem{theorem}{Theorem}[section]
\newtheorem{lemma}[theorem]{Lemma}
\theoremstyle{definition}
\newtheorem{definition}[theorem]{Definition}

\newtheorem{corollary}[theorem]{Corollary}

\theoremstyle{remark}
\newtheorem{remark}[theorem]{Remark}

\numberwithin{equation}{section}
\allowdisplaybreaks

\begin{document}
\title[A generation theorem for the perturbation of semigroups]{A generation theorem for the perturbation of strongly continuous semigroups by unbounded operators}

\author[X.-Q. Bui]{Xuan-Quang Bui}
\address{Faculty of Fundamental Sciences, PHENIKAA University, Hanoi 12116, Vietnam}
\email{quang.buixuan@phenikaa-uni.edu.vn}

\author[N.D. Huy]{Nguyen Duc Huy}
\address{VNU University of Education, Vietnam National University, Hanoi, 144 Xuan Thuy, Cau Giay, Hanoi, Vietnam} 
\email{huynd@vnu.edu.vn}

\author[V.T Luong]{Vu Trong Luong}
\address{VNU University of Education, Vietnam National University, Hanoi, 144 Xuan Thuy, Cau Giay, Hanoi, Vietnam} 
\email{vutrongluong@vnu.edu.vn}

\author[N.V. Minh]{Nguyen Van Minh}
\address{Department of Mathematics and Statistics, University of Arkansas at Little Rock, 2801 S University Ave, Little Rock, AR 72204, USA}
\email{mvnguyen1@ualr.edu}
\thanks{The second author's research is funded by Vietnam National University, Hanoi (VNU) under project number QG.23.47.}
\date{\today}
\begin{abstract} 
In this paper we study the well-posedness of the evolution equation of the form $u'(t)=Au(t)+Cu(t)$, $t\ge 0$, where $A$ is the generator of a $C_0$-semigroup and $C$ is a (possibly unbounded) linear operator in a Banach space $\mathbb{X}$. We prove that if $A$ generates a $C_0$-semigroup $\left (T_A(t)\right )_{t \geq 0}$ with $\|T_A(t)\| \le Me^{\omega t}$ in a Banach space $\mathbb{X}$ and $C$ is a linear operator in $\mathbb{X}$ such that $D(A)\subset D(C)$ and $\| CR(\mu, A)\| \le K/(\mu -\omega)$ for each $\mu>\omega$, then, the above-mentioned evolution equation is well-posed, that is, $A+C$ generates a $C_0$-semigroup $\left (T_{A+C}(t)\right )_{t \geq 0}$ satisfying $\| T_{A+C}(t)\| \le Me^{(\omega +MK)t}$. Our approach is to use the Hille-Yosida's theorem. We also define a linear space of linear operators $C$ in $\mathbb X$ associated with the operator $A$ with a norm defined as $\| C\|_A:=\frac{1}{M} \sup_{\mu >\omega } \| (\mu-\omega) CR(\mu,A)\| <\infty$, in which we show that the exponential dichotomy of $\left (T_A(t)\right )_{t \geq 0}$ persists under small perturbation $C$. The obtained results seem to be new.
\end{abstract}
\maketitle
\section{Introduction and Preliminaries}
\label{section 1}
In this paper we study the well-posedness of the evolution equation of the form 
\begin{equation}\label{eq1}
u'(t)=Au(t)+Cu(t), \quad t\ge 0,
\end{equation}
where $A$ is the generator of a $C_0$-semigroup and $C$ is a (possibly unbounded) linear operator in a Banach space $\mathbb{X}$. To be precise, we will find conditions on $A$ and $C$ so that $A+C$ generates a $C_0$-semigroup of linear operators in $\mathbb{X}$. We then go further studying the ``size'' of (possibly unbounded) perturbation $C$ under which an asymptotic behavior (exponential dichotomy) of the unperturbed equation persists.

Perturbation of strongly continuous semigroups is one of the central topics in the semigroup theory. We refer the reader to some classic results in \cite{bat}, \cite[p. 631--641]{dunsch}, \cite[Chapter III]{engnag}, \cite{cholei, hen, keiwei, meg, voi}
and their references for more information. The case of unbounded perturbation is of particular interest as it is very hard to measure the size of the perturbation. In this paper we will go further in this direction to study some conditions for the operator $A+C$ to generate a strongly continuous semigroup. Our approach is to use the Hille-Yosida's Theorem for the perturbed semigroups rather than using the Variation-of-Parameters Formula to prove the perturbed operator generates a $C_0$-semigroup. Given an operator $A$ that generates a $C_0$-semigroup we propose a norm for a linear space of (possibly unbounded) linear operators associated with $A$ (see Definition \ref{minh_distance} and Condition (ii) of Theorem~\ref{the main}) that can be used to measure the size of the perturbation for which the operator $A+C$ generates a $C_0$-semigroup and that can be used to study the persistence of exponential dichotomy under perturbation. We then discuss this ``size'' with the so call Yosida distance between linear operators that was introduced recently to study the perturbation of asymptotic behavior of $C_0$-semigroups. Our main result is Theorem~\ref{the main} that seems to be new.

In the paper we use some standard notations: $\mathbb{R},\,\mathbb{C}$ stand for the fields of real and complex numbers, respectively. By $\mathbb{X}$ we often denote a Banach space over $\mathbb{C}$ with norm $\|\cdot\|$. The space $\mathcal{L}(\mathbb{X})$ of all bounded linear operators on $\mathbb{X}$ with norm $\| \cdot \|$, by abuse of notation for our convenience if this does not cause any confusion. For a linear operator $A$ in $\mathbb{X}$ we denote by $D(A)$ its domain, and $\sigma (A)$ and $\rho(A)$ its spectrum and resolvent set, respectively. If $\mu\in \rho(A)$, then $R(\mu,A):=(\mu-A)^{-1}$. We recall the Hille-Yosida's Theorem for the generation of strongly continuous semigroups that we will use later on (see \cite[Theorem~5.3, p.~20]{paz}, \cite{yos}).

\begin{theorem}[Hille-Yosida's Theorem]
A linear operator $A$ is the infinitesimal generator of a $C_0$-semigroup $\left (T_A(t)\right)_{t \geq 0}$ satisfying $\| T_A(t)\| \le Me^{\omega t}$, if and only if
\begin{enumerate}
\item $A$ is closed and $D(A)$ is dense in $\mathbb{X}$;
\item the resolvent set $\rho (A)$ of $A$ contains the ray $(\omega, \infty)$ and
\begin{equation}\label{1.1}
\| R(\lambda, A)^n\| \le \frac{M}{(\lambda -\omega)^n},
\quad 
\mbox{ for all } \lambda >\omega, \qquad n=1,2,\dots.
\end{equation}	
\end{enumerate}
\end{theorem}
Note that when $M=1$ the inequality \eqref{1.1} can be simplified to be
\begin{equation}
\| R(\lambda, A)\| \le \frac{1}{\lambda -\omega}, \quad \mbox{ for all } \lambda >\omega .
\end{equation}
\section{Main Results}
To prove the main result (Theorem~\ref{the main}) we need the following lemmas that deal with several special cases.
\begin{lemma}\label{lem new}
Let $A$ be the generator of a $C_0$-semigroup $\left (T_A(t)\right )_{t\ge 0}$ in $\mathbb X$ such that for all $t\ge 0$ 
$$
\|T_A(t)\| \le e^{\omega t},\quad t\ge 0,
$$
where $\omega $ is a given real number. Assume further that $C$ is a linear operator in $\mathbb X$ such that 
\begin{enumerate}
\item $D(A)\subset D(C)$;
\item There exist a constant $K>0$ such that if $\mu>\omega$, then
\begin{equation}\label{2.1}
\| CR(\mu,A)\| \le \frac{K}{\mu-\omega} .
\end{equation}
\end{enumerate}
Then, $A+C$ generates a $C_0$-semigroup, denoted by $\left (T_{A+C}(t)\right )_{t\ge 0}$ that satisfies
\begin{align}\label{2.2}
\| T_{A+C}(t)\| \le e^{\left (\omega +K\right )t}, \quad t \ge 0.
\end{align}
\end{lemma}
\begin{proof}
First, we will prove that
\begin{equation}\label{2.4}
 \rho (A+C) \supset (K+\omega,\infty) .
\end{equation}
For each $\mu >K+\omega$, as $K>0$ we have $\mu\in \rho(A)$, and
\begin{align}
(\mu-(A+C))R(\mu,A) 
&= 
\mu R(\mu, A)-AR(\mu, A)- CR(\mu, A)
\notag \\
&= 
I- CR(\mu, A) 
\label{100}
\end{align}
so,
\begin{align}
(\mu-(A+C))
&= \left[I- CR(\mu,A)\right] (\mu-A) \label{109}.
\end{align}
But, for $\mu > K+\omega$ we have $K/(\mu-\omega ) <1$, so $\| CR(\mu, A)\| <1$, and then, the operator $ I- CR(\mu,A) $ is invertible. Thus, the operator $(\mu-(A+C))$ is invertible as well, and \eqref{2.4} follows. In particular, this also yields that $A+C$ is a closed operator (as $\rho (A+C)\not=\emptyset$).

Next, for each $\mu >K+\omega $, it is easy to verify the identity
\begin{equation}\label{3.6}
R(\mu,A+C)-R(\mu,A)=R(\mu,A+C)CR(\mu,A).
\end{equation}
Hence,
\begin{equation}
R(\mu,A+C) \left[ I-CR(\mu,A)\right] =R(\mu,A).
\end{equation}
By assumption \eqref{2.1}, for $\mu>K+\omega$ as $\|CR(\mu,A)\| \le K/(\mu-\omega)<1$, the operator $\left[ I-CR(\mu,A)\right] $ is invertible and 
\begin{equation}
\| \left[ I-CR(\mu,A)\right] ^{-1}\| \le \frac{1}{1-K/(\mu-\omega)}.
\end{equation}
Consequently, for all $\mu>K+\omega$
\begin{align}
\| R(\mu,A+C)\| 
& =
\left \| R(\mu,A)\left[ I-CR(\mu,A)\right] ^{-1}\right \| 
\notag \\
&\le
\frac{1}{\mu-\omega } \cdot \frac{1}{1-K/(\mu-\omega)}\notag\\
&\le 
\frac{1}{\mu-(\omega+K)}.
\label{2.8}
\end{align}
By the Hille-Yosida's generation theorem, since $D(A)\subset D(A+C)$ is densely everywhere in $\mathbb X$, \eqref{2.8} yields that the closed linear operator $A+C$ generate a $C_0$-semigroup, denoted by $\left (T_{A+C}(t)\right )_{t\ge 0}$ that satisfies \eqref{2.2}. And, 
the proof is complete.
\end{proof}

\begin{lemma}\label{lem 2}
Let $A$ be the generator of a $C_0$-semigroup $\left (T_A(t)\right )_{t\ge 0}$ in $\mathbb X$ such that
$$
\| T_A(t)\| \le M,\quad t\ge 0,
$$
where $M\ge 1 $ is a given real number. Assume further that $C$ is a linear operator in $\mathbb X$ such that 
\begin{enumerate}
\item $D(A)\subset D(C)$;
\item there exists a constant $K>0$ such that if $\mu>0$, then
\begin{equation}
\| CR(\mu,A)\| \le \frac{K}{\mu} .
\end{equation}
\end{enumerate}
Then, $A+C$ generates a $C_0$-semigroup, denoted by $\left (T_{A+C}(t)\right )_{t\ge 0}$ that satisfies
\begin{align}\label{2.12}
\| T_{A+C}(t)\| \le M e^{MK t}, \quad t \ge 0.
\end{align}
\end{lemma}
\begin{proof}
We re-norm $\mathbb{X}$ by using 
$$
|x| :=\sup_{t\ge 0} \| T_A(t)x\|, \quad x\in \mathbb{X}.
$$
Then, as is well known, for all $x\in\mathbb{X}$,
\begin{equation}\label{2.13}
\| x\| \le |x| \le M\| x\| 
\end{equation}
and
$\left (T_A(t)\right )_{t\ge 0}$ is a $C_0$-semigroup of contractions in $(\mathbb{X}, |\cdot |)$. Next, for each $\mu >0$ we have
\begin{align}\label{2.14}
| CR(\mu,A)x| &\le M \| CR(\mu,A)x\| 
\notag \\
&\le 
M\cdot \frac{K}{\mu}\| x\| 
\notag \\
&\le 
\frac{MK}{\mu} |x| .
\end{align}
By Lemma \ref{lem new}, the closed operator $A+C$ generates a $C_0$-semigroup in $(\mathbb{X},|\cdot|)$ that satisfies
$$
|T_{A+C}(t) | \le e^{MKt},\quad t\ge 0.
$$
Therefore, by using \eqref{2.13}
\begin{align}
 \| T_{A+C}(t)x\| \le	|T_{A+C}(t)x | 
& \le e^{MKt}|x|
\le Me^{MKt} \| x\| 
\end{align}
for all $x\in\mathbb X$. And, the lemma is proved.
\end{proof}
\begin{remark}\label{rem 2.3}
If $C$ is a bounded linear operator, the estimate \eqref{2.12} can be improved a little bit. In fact, in this case $K$ can be replaced with $M\| C\|$. Next, in \eqref{2.14} we can improve the estimate as follows: Since the semigroup $\left (T_{A}(t)\right )_{t\ge 0}$ in $(\mathbb{X},|\cdot |)$ is an contraction semigroup, $|R(\mu,A)| \le 1/\mu $, so
\begin{align}\label{2.14'}
| CR(\mu,A)x| &\le |C| \cdot |R(\mu,A)x|	\le \frac{|C|}{\mu} |x|.
\end{align}
Therefore, by Lemma~\ref{lem new},
\begin{align}
\| T_{A+C}(t)x\| 
\le	|T_{A+C}(t)x | 
\le e^{|C|t}|x|
\le Me^{|C|t} \| x\| .
\end{align}
We have
\begin{align*}
|Cx| :=\sup_{t\ge 0}\| T_A(t)Cx\|
&= M\| Cx\| 
\\
& \le M\|C\| |x|.
\end{align*}
That means, $|C| \le M\| C\|$, and thus,
\begin{align}
\| T_{A+C}(t)\| \le	
 Me^{M\| C\|t} .
\end{align}
\end{remark}

For the general $C_0$-semigroup $\left (T_A(t)\right )_{t\ge 0}$ we have the following theorem that is the main result of the paper:
\begin{theorem}\label{the main}
Let $A$ be the generator of a $C_0$-semigroup $\left (T_A(t)\right )_{t\ge 0}$ in $\mathbb X$ such that for all $t\ge 0$ 
$$
\| T_A(t)\| \le M e^{\omega t},\quad t\ge 0,
$$
where $\omega$ is a given real number. Assume further that $C$ is a linear operator in $\mathbb X$ such that 
\begin{enumerate}
\item $D(A)\subset D(C)$;
\item There exists a constant $K>0$ such that if $\mu>\omega$, then
\begin{equation}
\| CR(\mu,A)\| \le \frac{K}{\mu-\omega} .
\end{equation}
\end{enumerate}
Then, $A+C$ generates a $C_0$-semigroup, denoted by $\left (T_{A+C}(t)\right )_{t\ge 0}$ that satisfies
\begin{align}\label{2.22}
\| T_{A+C}(t)\| \le M e^{(\omega +MK) t}, \quad t \ge 0.
\end{align}
\end{theorem}
\begin{proof}
Set $S(t):=e^{-\omega t}T_A(t)$. The generator of $\left (S(t)\right )_{t \geq 0}$ is $B:=A-\omega,$ and $\| S(t)\| \le M$ for all $t\ge 0$. We now apply 
Lemma \ref{lem 2} to the semigroup $\left (S(t)\right )_{t \geq 0}$. In fact, for $\mu >0$, we have
\begin{align}
\| CR(\mu, B)\| &= \| CR(\mu+\omega, A)\| \notag \\
& \le \frac{K}{\mu +\omega-\omega}
\notag\\
& = 
\frac{K}{\mu}.
\end{align}
Therefore, by Lemma \ref{lem 2}, the operator $B+C$ generates a $C_0$-semigroup $\left (T_{B+C}(t)\right )_{t \geq 0}$ in $\mathbb X$. It is easy to see that the semigroup $\left (e^{\omega t} T_{B+C}(t)\right )_{t \geq 0}$ has $B+C+\omega= A+C$ as its generator. That is, the $C_0$-semigroup $\left (e^{\omega t} T_{B+C}(t)\right )_{t \geq 0}$ is generated by $A+C$, so $(e^{\omega t} T_{B+C}(t))_{t\ge 0}=(T_{A+C}(t))_{t\ge 0}$ which satisfies [\eqref{2.2} applies to $B+C$]:
\begin{align*}
\| T_{A+C}(t)\| &=e^{\omega t}T_{B+C}(t)\| \\
&= e^{\omega }\| T_{B+C}(t)\|\\
&\le e^{\omega t}Me^{MKt}, \quad t\ge 0 .
\end{align*}
And, the theorem is proved.
\end{proof}
\begin{remark}
If $C$ is a bounded linear operator, the estimate \eqref{2.22} can be improved a little bit. In fact, in this case $K$ can be replaced with $M\| C\|$, therefore, by Remark~\ref{rem 2.3}, we get
\begin{align*}
\| T_{A+C}(t)\| & =e^{\omega t}\| T_{B+C}(t)\| \notag \\
& \le e^{\omega t} Me^{M\| C\| t}\notag \\
&= Me^{(\omega +M\| C\|)t},  \quad t\ge 0  .
\end{align*}
\end{remark}
\section{About the Space $\mathcal{GL}_A(\mathbb{X}) $}
In this section, given an operator $A$ as the generator of a $C_0$-semigroup we will define a normed space $\mathcal{GL}_A(\mathbb{X})$ associated with $A$, that is equipped with a norm determined by Theorem~\ref{the main}. We will study the persistence of some behavior of solutions of the equation $u'(t)=Au(t)$, $t\ge 0,$ like exponential dichotomy under small perturbation of $A$ in $\mathcal{GL}_A(\mathbb{X})$.

Below we always assume that $A$ generates a $C_0$ semigroup $\left (T_{A}(t)\right )_{t \geq 0}$ with $\left \|T_A(t)\right \| \le Me^{\omega t}$ for all $t\ge 0$ and for certain fixed constants $M \ge 1$ and $\omega \in \mathbb R$.
\begin{definition}
We define $\mathcal{GL}_A(\mathbb{X})$ to be the linear space of all linear operators $C$ in $\mathbb{X}$ with $D(C)= D(A)$ and 
\begin{equation}
\sup_{\mu >\omega }(\mu -\omega)\| CR(\mu ,A)\|  < \infty.
\end{equation}
\end{definition}
For each $C\in \mathcal{GL}_A(\mathbb{X})$ we define
\begin{equation}\label{minh_distance}
\| C\|_A:=\frac{1}{M} \sup_{\mu >\omega } \| (\mu-\omega) CR(\mu,A)\| <\infty.
\end{equation}
\begin{lemma}
Given an operator $A$ in $\mathbb X$ as the generator of a $C_0$-semigroup. Then, $( \mathcal{GL}_A(X),\|\cdot\|_A)$ is a normed space.
\end{lemma}
\begin{proof}
It is easy to see that $\mathcal{GL}_A(\mathbb{X})$  is a linear space, and that $\| \cdot \|$ satisfies the following conditions:
\begin{align}
& \| C\|_A \ge 0,\\
& \| \mu C\|_A = |\mu | \| C\|_A, \\
& \| C+B\|_A \le \| C\|_A +\| B\|_A, 
\end{align}
for all $B,\, C\in \mathcal{GL}_A(\mathbb{X})$, and $\mu \in \mathbb{C}$. We just need to show that $\|C\|_A=0$ implies that $C=0$ (on $D(A)$). In fact, if $\|C\|_A=0$, then for every $x\in \mathbb{X}$ and $\mu>\omega$ we have $CR(\mu,A)x=0$. Since $R(\mu,A)$ is a bijective map from $\mathbb{X}$ on $D(A)$, for every given $y\in D(A)$ there exists an $x\in \mathbb{X}$ such that $y=R(\mu,A)x$. Therefore, for each $y\in D(A)$, $Cy=CR(\mu,A)x=0$. The lemma is proved.
\end{proof}
\begin{remark}
Notice that if $C\in \mathcal{L}(\mathbb{X})$, then, $\| C\| _A\le \| C\|$. In fact, by the Hille-Yosida's theorem (applied to $A$ and the $C_0$-semigroup $\left (T_A(t)\right )_{t \geq 0}$),
$$
\frac{1}{M} \sup_{\mu >\omega }
\| (\mu-\omega) R(\mu,A)\|
\le 1,
$$
so we have
\begin{align}
\| C\|_A &=\frac{1}{M} \sup_{\mu >\omega } \| (\mu-\omega) CR(\mu,A)\| \notag \\
&\le \| C\| \frac{1}{M} \sup_{\mu >\omega } \| (\mu-\omega) R(\mu,A)\|
\notag \\
&\le 
\| C\| .
\end{align}	
This means that $\mathcal L (\mathbb X) \subset \mathcal{GL}_A(\mathbb{X})$.
\end{remark}
Below we will apply Theorem~\ref{the main} to study the persistence of an asymptotic behavior of the $C_0$-semigroup $\left (T_{A}(t)\right )_{t\ge 0}$ when its generator $A$ is under small perturbation $C$ in the sense of the norm in  $\mathcal{GL}_A(\mathbb{X})$. This can be done via the concept of \textit{Yosida distance}.

Recall that given two linear operators $A$ and $B$ in $\mathbb{X}$, the \textit{Yosida distance} $d_Y(A,B)$, assuming that $\rho (A)$ and $ \rho(B)$ contain the ray $[\omega,\infty)$ for some $\omega \in \mathbb{R}$, is defined as,
see \cite{buimin, buimin2},
\begin{equation}
d_Y(A,B):= \limsup_{\lambda\to +\infty} \lambda^2\| R(\lambda,A)-R(\lambda, B)\| .
\end{equation}

\begin{theorem}\label{the yosida}
Let $A$ be the generator of a $C_0$-semigroup. Then, for any $C_1,\,C_2 \in \mathcal{GL}_A(\mathbb{X})$ the following assertion is true:
\begin{equation}
d_Y(A+C_1,A+C_2) \le \| C_1-C_2\|_A .
\end{equation}
\end{theorem}
\begin{proof} By Theorem~\ref{the main}, both $A+C_1$ and $A+C_2$ are the generators of $C_0$-semigroups, so, for sufficiently large $\mu_0>0$, $[\mu_0,\infty) \subset \rho(A+C_i)$, $i=1,2$. Then, for $\mu>\mu_0$, 
by \eqref{3.6} we have
\begin{align*}
& R(\mu,A+C_1)-R(\mu,A+C_2) 
\\
& = R(\mu,A+C_1)C_1R(\mu,A) - R(\mu,A+C_2)C_2R(\mu,A) .
\end{align*}
Therefore, by part (i) and the Hille-Yosida's Theorem for $C_0$-semigroups,
\begin{align}
&
\| R(\mu,A+C_1)-R(\mu,A+C_2)\|
\notag \\
& \le 
\| R(\mu,A+C_1)\| \cdot \| C_1R(\mu,A)-C_2R(\mu, A)\| \notag \\
& \qquad 
+ 
\| R(\mu,A+C_1)-R(\mu,A+C+C_2)\| \cdot \| C_2R(\mu,A)\| \notag \\
&\le \frac{M}{\mu-(\omega +M^2\| C_1\|_A)} \| (C_1-C_2)R(\mu,A)\| 
\notag \\
& \qquad 
+ 
\| R(\mu,A+C_1)-R(\mu,A+C_2)\| \cdot \| C_2R(\mu,A)\| . 
\end{align}
Therefore, for large $\mu$, by assertion (i),
\begin{align}
& \mu^2\| R(\mu,A+C_1)-R(\mu,A+C_2)\| (1- \| C_2R(\mu,A)\|) 
\notag \\
& \le \frac{\mu^2M}{\mu-(\omega +M^2\| C_1\|_A)} \| (C_1-C_2)R(\mu,A)\|.
\end{align}
Consequently, by the definition of $\|\cdot \|_A$, for $ \mu >\omega$, we have
\begin{align*}
&
\| C_2R(\mu,A)\| \le \frac{1}{M} \cdot \frac{\|C_2\|_A}{\mu-\omega}, \\
&
\| (C_1- C_2)R(\mu,A) \| \le \frac{1}{M} \cdot \frac{\|C_1-C_2\|_A}{\mu-\omega},
\end{align*}
so, for sufficiently large $\mu>0$, 
\begin{align}
& \mu^2\| R(\mu,A+C_1)-R(\mu,A+C_2)\| \left(1-\frac{\|C_2\|_A}{\mu-\omega}\right) \notag \\
& \le \frac{\mu^2}{\mu-(\omega +M^2\| C_1\|_A)}\cdot \frac{\|C_1-C_2\|_A}{\mu-\omega} .
\end{align}
Finally,
\begin{align}
d_Y(A+C_1,A+C_2) 
&= 	\limsup_{\mu\to\infty} \mu^2\| R(\mu,A+C_1)-R(\mu,A+C_2)\| \notag  \\
&=	\limsup_{\mu\to\infty}\left[ \mu^2\| R(\mu,A+C_1)-R(\mu,A+C_2) \notag  \| \left(1-\frac{\|C_2\|_A}{\mu-\omega}\right) \right] \\
&\le 	\limsup_{\mu\to\infty}\frac{\mu^2}{\mu-(\omega +M+\| C_1\|_A)}\cdot \frac{\|C_1-C_2\|_A}{\mu-\omega} \notag  \\
&= \|C_1-C_2\|_A . \label{3.13}
\end{align}
The theorem is proved.
\end{proof}

\begin{definition}[Exponential dichotomy]\label{TM-Definition1.2}
A linear $C_0$-semigroup $(T(t))_{t \ge 0}$ in a Banach space $\mathbb X$ is said to have an \textit{exponential dichotomy} or to be \textit{hyperbolic} if there exist a bounded projection $P$ on $\mathbb{X}$ and positive constants $N$ and $\alpha$ satisfying
\begin{enumerate}
\item $T(t) P = P T(t)$, for $t \ge 0$;
\item $T(t)\big|_{\ker (P)}$ is an isomorphism from $\ker (P)$ onto $\ker (P)$, for all $t \ge 0$, 
and its inverse on $\ker (P)$ is defined by $T(-t):=\left (T(t)|_{\ker (P)}\right )^{-1}$;
\item the following estimates hold 
\begin{align}
&
\|T(t)x\| \le N e^{-\alpha t} \|x\|,
\quad \mbox{ for all }
t \ge 0,\quad
x \in \mathrm{Im}(P),
\\
&
\|T(-t) x\| \le N e^{-\alpha t} \|x\|,
\quad \mbox{ for all }
t \ge 0,\quad
x \in \ker (P).
\end{align}
\end{enumerate}
When the projection $P=I$ the semigroup $\left (T(t)\right )_{t \ge 0}$ is said to be {\it exponentially stable}.
\end{definition}
\begin{lemma}\label{lem engnag}
Let $(T(t))_{t \ge 0}$ be a $C_0$-semigroup in a Banach space $\mathbb X$. Then, 
\begin{enumerate}
	\item 	It has an exponential dichotomy if and only if 
	$$\sigma (T(1))\cap \{ z\in \mathbb C: |z| =1\}=\emptyset .$$
	\item It is exponentially stable if and only if $r_\sigma (T(1))<1.$
\end{enumerate}
\end{lemma}	
\begin{proof}
For the proof of Part (i), see \cite[Theorem 1.17]{engnag}. For Part (ii) the proof is straightforward via the Spectral Radius Theorem.
\end{proof}

\begin{corollary}
Let $A$ be the generator of a $C_0$-semigroup $\left (T_{A}(t)\right )_{t\ge 0}$ in $\mathbb{X}$ satisfying
$ \| T_A(t)\| \le Me^{\omega t}$. Assume further that for a given $C_1\in \mathcal{GL}_A(\mathbb{X})$, the semigroup $\left (T_{A+C_1}(t)\right )_{t\ge 0}$ generated by $A+C_1$ has an exponential dichotomy. 	
Then, for any $C_2 \in \mathcal{GL}_A(\mathbb{X})$, the semigroup $\left (T_{A+C_2}(t)\right )_{t\ge 0}$ generated by $A+C_2$ has also an exponential dichotomy provided that  $\|C_1-C_2\|_A$ is sufficiently small. In particular, if $\left (T_{A+C_1}(t)\right )_{t \ge 0}$ is exponentially stable, then $(T_{A+C_2}(t))_{t \ge 0}$ is also exponentially stable provided that $\| C_1-C_2\|_A$ is sufficiently small.
\end{corollary}
\begin{proof}
By \eqref{3.13} in the proof of Theorem~\ref{the yosida} the Yosida distance  $d_Y(A+C_1,A+C_2)$ is sufficiently small if $\| C_1-C_2\|_A$ is sufficiently small. On the other hand, by the proof of  \cite[Theorem 3.2]{buimin2}, we have the estimate
\begin{align}
\left \| T_{A+C_1}(t)-T_{A+C_2}(t)\right \| 
&
\le tM^2 e^{4\omega_0 t}d_Y(A+C_1,A+C_2) \notag\\
& \le tM^2 e^{4\omega_0 t}\| C_1-C_2\|_A ,\label{be}
\end{align}
where 
$$
\omega_0 := \omega + M^2 (\| C_1\| +\epsilon_0),
$$
with $\epsilon_0 >0$ and $\| C_1-C_2\| < \epsilon_0.$ Suppose that $\left (T_{A+C_1}(t)\right )_{t\ge 0}$ is exponentially stable. Then, there exists a positive integer $n_0$ such that $\| T_{A+C_1}(n_0) \| < 1$. For sufficiently small $\|C_1-C_2\|$, by \eqref{be} we have $\| T_{A+C_2}(n_0)\|  < 1$ as well. Consequently, $(T_{A+C_2}(t))_{t \ge 0}$ is exponentially stable.

Now suppose that $(T_{A+C_2}(t))_{t \ge 0}$ has an exponential dichotomy. 
By \eqref{be}, we have
\begin{equation}
\| T_{A+C_1}(1)-T_{A+C_2}(1)\| \le  M^2 e^{4\omega_0}\| C_1-C_2\|_A.
\end{equation}
Therefore, if  $\| C_1-C_2\| $ is sufficiently small, $\| T_{A+C_1}(1)-T_{A+C_2}(1)\| $ is also sufficiently small. By \cite[Theorem 2.1, p. 16]{dalkre}, as the set $\Gamma:= \mathbb C \backslash 
\{ z\in \mathbb C: |z| =1\}$ is open
for sufficiently small $\| T_{A+C_1}(1)-T_{A+C_2}(1)\|$, 
$\sigma (T_{A+C}(1)) \subset \Gamma$, that is
$$
\sigma ((T_{A+C_2}(1)) \cap \{ z\in \mathbb C: |z| =1\}=\emptyset .
$$
By Lemma \ref{lem engnag}, this yields that $(T_{A+C_2}(t))_{t \ge 0}$ has an exponential dichotomy. 
\end{proof}

\section*{Acknowledgements}
The authors thank the anonymous referee for their careful reading the manuscript and helpful suggestions to correct some calculations as well as improving the presentation of the paper.

 
\end{document}